\documentclass{amsart}
\usepackage[pdftex]{graphicx}
\usepackage{sidecap}
\usepackage{subfig}

\usepackage{fancyhdr}
\theoremstyle{plain}
\newtheorem{theorem}{Theorem}[section]
\newtheorem{lemma}[theorem]{Lemma}
\newtheorem{proposition}[theorem]{Proposition}
\newtheorem{corollary}[theorem]{Corollary}
\theoremstyle{definition}
\newtheorem{definition}[theorem]{Definition}

\newtheorem{example}[theorem]{Example}
\theoremstyle{definition}
\theoremstyle{remark}
\newtheorem{remark}{Remark}

\makeatletter
\def\@seccntformat#1{\csname the#1\endcsname.\quad}
\def\section{\@startsection{section}{1}{\z@}%
{-3.5ex \@plus -1ex \@minus -.2ex}%
{2.3ex \@plus.2ex}%
{\Large\bf}}
\def\subsection{\@startsection{subsection}{2}{\z@}%
{-3.25ex\@plus -1ex \@minus -.2ex}%
{1.5ex \@plus .2ex}%
{\bf\large}}
\def\subsubsection{\@startsection{subsubsection}{3}{\z@}%
{-3.25ex\@plus -1ex \@minus -.2ex}%
{1.5ex \@plus .2ex}%
{\normalfont\normalsize\bf}}

\def\blfootnote{\xdef\@thefnmark{}\@footnotetext}
\makeatother

 \newcommand{\be}{\begin{equation}}
 \newcommand{\ee}{\end{equation}}
 \newcommand{\bd}{\begin{displaymath}}
 \newcommand{\ed}{\end{displaymath}}
 \newcommand{\bea}{\begin{eqnarray}}
 \newcommand{\eea}{\end{eqnarray}}
 \newcommand{\beas}{\begin{eqnarray*}}
 \newcommand{\eeas}{\end{eqnarray*}}
 \newcommand{\bc}{\begin{center}}
 \newcommand{\ec}{\end{center}}

 \title{Numerical Algorithms  for a Variational problem of the Spatial
Segregation of Reaction-diffusion Systems}

\author{Avetik Arakelyan}
\address{Institute of Mathematics, National Academy of Science of Armenia,
 0019  Yerevan, Armenia}
\email{avetik@math.kth.se}

\author{Farid Bozorgnia*}
\address{Faculty of Sciences, Persian Gulf University, Boushehr 75168, Iran}
\email{faridb@kth.se}
\thanks{*F. Bozorgnia was supported by a Grant No. 89350023 from IPM.
Corresponding author. Address: School of Mathematics, Institute for Research in Fundamental Sciences (IPM), P.O. Box 19395-5746, Tehran, Iran.}

\keywords{ Free boundary problems, Segregation,  Reaction-diffusion Systems, Finite Difference.}
\begin{document}
\begin{large}
\maketitle
\renewcommand{\theequation}{1.\arabic{equation}}
\begin{abstract}
This paper is concerned with the numerical approximation of a class of stationary states for reaction-diffusion system with $m$  densities having disjoint support,  which are governed by a minimization problem. We use  quantitative properties of both, solutions and free boundaries, to derive our scheme. Furthermore, the proof of convergence of the numerical method is given in some particular cases. The proposed numerical scheme is applied   for the spatial segregation limit of diffusive Lotka-Volterra models in presence of high  competition and inhomogeneous Dirichlet boundary conditions. The numerical implementations of the resulting approach are discussed  and computational tests are presented.
\end{abstract}

\section{Introduction}
 \quad In recent years there have been intense studies of spatial segregation for reaction-diffusion systems. The existence of spatially inhomogeneous solutions for competition models of Lotka-Volterra type in the case of two and more competing densities  have been considered  \cite{MR2146353,MR2151234, MR2283921,MR2300320,MR1459414, MR1900331, MR2417905}.
 The objective of  this paper is to study numerical solutions of two classes  of possible segregation states. The first class is  related with an arbitrary number  of competing densities, which are governed by a minimization problem.

Let $\Omega  \subset \mathbb{R}^n, (n\geq 2)$ be a connected and  bounded domain with smooth boundary, and  $m$ be a fixed integer.
We consider the steady-states of $m$ competing species coexisting in  the same area $\Omega$.
Let $u_{i}(x)$ denotes  the population density of the  $i^\textrm{th}$ component with the internal dynamic  prescribed by $f_{i}(x)$. Here  we assume that $f_i$ is uniformly  continuous   and    $f_{i}(x) \geq 0.$

  The $m$-tuple $U=(u_1,\cdots,u_m)\in (H^{1}(\Omega))^{m},$ is called \emph{segregated state} if
\[
u_{i}(x) \cdot  u_{j}(x)=0,\  \text{a.e. } \text{ for  } \quad i\neq j,  \ x\in \Omega.
\]

\textbf{Problem (A):}
Consider the following minimization problem
 \begin{equation}\label{1}
\text{  Minimize  }  E(u_1, \cdots, u_m)=\int_{\Omega}  \sum_{i=1}^{m} \left( \frac{1}{2}| \nabla u_{i}|^{2}+f_{i}u_{i}  \right) dx,
\end{equation}
  over the set
  $$S={\{(u_1,\dots,u_{m})\in (H^{1}   (\Omega))^{m} :u_{i}\geq0, u_{i} \cdot u_{j}=0, u_{i}=\phi_{i} \quad \text {on} \quad \partial  \Omega}\},$$
where $\phi_{i} \in H^{\frac{1}{2}}(\partial \Omega),$\; $\phi_{i}  \cdot \phi_{j}=0,$ for $i\neq j$ and $ \phi_{i}\geq 0$ on the boundary $\partial \Omega.$ We assume that $f_i$ is uniformly  continuous   and    $f_{i}(x) \geq 0.$

\textbf{Problem (B): }
Our second  problem,  which appears in the study of population ecology, is the case when
 high  competitive interactions between different species occurs.
 As the rate of interaction of two different species goes to infinity, the  competition-diffusion system
shows a limiting configuration with segregated state.
We refer the reader to \cite{MR2151234,MR2363653, MR2079274, MR1303035, MR1459414, MR1687440, MR1900331}  and in particular to \cite{MR1687440} for models involving Dirichlet boundary data. A complete analysis of the stationary case has been studied  in \cite{MR2151234}. Also  numerical simulation   for the spatial segregation limit of two diffusive Lotka-Volterra models in presence of strong competition and inhomogeneous Dirichlet boundary conditions is provided in \cite{MR2459673}.  In \cite{MR2459673} the authors solve the problem for small  $\varepsilon$ and then let $\varepsilon  \longrightarrow 0 ,$ while in our work we use the qualitative properties of the limiting problem. Unlike the results in  \cite{MR2459673}, where the authors provide only simulations of their proposed algorithm, we give a numerical consistent variational system with strong interaction, and provide disjointness condition of populations during the iteration of the scheme. Moreover, by discussing these two problems we show that the proposed idea can be generalized for two or more species that competing each other.

 Let $d_{i}, \lambda $ be positive numbers. Consider the following system of $m$ differential equations

\begin{equation}
\left \{
\begin{array}{lll}\label{2}
-d_{i}\Delta u_{i} = \lambda u_{i}(1 - u_{i})-\frac{1}{\varepsilon} u_{i} \sum_{j\neq i} u^{2}_{j}   & \text{ in } \Omega,\\
    u_{i}(x, y) = \phi_{i}(x, y) & \text{ on } \partial \Omega,\\
\end{array}
\right.
\end{equation}
for $ i=1,\cdots ,m,$  where $\phi_{i} \in H^{\frac{1}{2}}(\partial \Omega)$ and $\phi_{i}  \cdot \phi_{j}=0, \, \phi_{i}\geq 0$ on the boundary $\partial \Omega.$  Our aim is to present a numerical approximation for this system as $\varepsilon \rightarrow 0.$ This system can be viewed as a steady state of the following auxiliary system in the case that the boundary values are time independent:

\begin{equation}
\left \{
\begin{array}{llll}
\frac{d}{dt} u_{i}- d_{i} \Delta u_{i} = \lambda u_{i}(1 - u_{i})-\frac{1}{\varepsilon} u_{i} \sum_{j\neq i} u^{2}_{j}   & \text{ in } \Omega \times (0, \infty ),\\
    u_{i}(x, y, t) = \phi_{i}(x, y) & \text{ on } \partial \Omega \times (0, \infty ),\\
     u_{i}(x, y, 0) = u_{i,0}(x, y)   & \text{ in } \Omega,\\
\end{array}
\right.
\end{equation}
for $i=1,\cdots ,m.$\\
One of the interesting results which  relates these  two problems   is given in \cite{MR2146353}.  Consider  the following  reaction-diffusion system of three competing species:
\begin{equation}\label{3}
\Delta  u_i = \frac{1}{\varepsilon} u_i \sum_{j\neq i} u_j,  \quad  u_i\geq 0, \text{ in } \Omega,\ u_i=\phi_i,\ \  \text{ on  } \partial \Omega \    i = 1, 2, 3,
\end{equation}
where we have  the same assumptions  on the  boundary values $\phi_i.$
 In \cite{MR2146353} it was shown the  uniqueness of the limiting configuration as $\varepsilon \rightarrow 0$ on a planar domain, with appropriate boundary conditions. Moreover, it was shown  that the corresponding minimization problem admits a unique solution, and the limiting configuration minimizes the following energy
\[
 \int_{\Omega}  \sum_{i=1}^{3}  \frac{1}{2}| \nabla u_{i}|^{2} dx,
\]
 over the set $S={\{u_{i}\in H^{1}   (\Omega):u_{i}\geq0, u_{i} \cdot u_{j}=0, u_{i}=\phi_{i} \quad \text {on} \,  \partial  \Omega,}\; i=1,2,3 \}.$
  For the numerical approximation of the  system (\ref{3}) the interested reader is referred to \cite{MR2563520}.

\section{Basic facts   for Problem (A) }
\renewcommand{\theequation}{2.\arabic{equation}}
\setcounter{equation}{0}
In  this section we will see that the solution of problem (\ref{1}) satisfies  a free boundary problem. In order to prove the existence of the minimizer we apply  the following classical theorem due to  \cite{MR1078018}.

\begin{theorem}\label{5}
 Let $V$ be  a reflexive Banach space with norm  $\| \cdot\|,$  and  $M\subset V$ be a weakly closed subset of $V$. Suppose $E : M\rightarrow \mathbb{R}$ is coercive on $M$ with respect to $V,$ that is
\begin{itemize}
\item [i)]  $ E(u)\rightarrow \infty $  as $\|u\|\rightarrow \infty, $  $u\in M$
and  $E$ is weakly lower semi-continuous on $M$ with respect to $V,$ that is
\item [ii)] for any $u\in M,$ any sequence $(u_{m})$ in $M$ such that $u_{m}\rightharpoonup u$ weakly in $V$ there holds
    $     E(u) \leq  \underset{m\rightarrow \infty} {\liminf} E(u_{m}).$
    \end{itemize}
 Then $E$ is bounded from below on $M$ and attains  its minimum    in $M.$
 \end{theorem}
Then we have the following existence and uniqueness result.
\begin{proposition}
Under the assumptions in Problem (A), there exist a  minimizer to \eqref{1}, and it is unique.
\end{proposition}
\begin{proof}
It is easy to see that the functional $E(u_1, \cdots, u_m)$ is coercive over the closed set $S,$ and lower semi-continuous on $S$ with respect to the space $(H^1(\Omega))^m.$ Thus the existence follows directly from above mentioned Theorem \ref{5}.
For the proof of uniqueness we  are using the same arguments as in \cite[Theorem $4.1$]{MR2151234}.
Suppose  that there exist  two different minimizers
 $U=(u_1,\cdots ,u_{m})$ and $V=(v_1,\cdots ,v_{m})$  of  \eqref{1}  such that
\be
E(u_1, \cdots , u_m)=E(v_1, \cdots , v_m)=c.
\ee
Define new functions $\overline{u}_{i}, \overline{v}_{i}$ by
\begin{equation*}
 \overline{u}_{i}(x)=u_{i}(x)-\sum_{k \neq i}u_{k}(x),
\end{equation*}
\begin{equation*}
\overline{v}_{i}(x)=v_{i}(x)-\sum_{k \neq i}v_{k}(x),
\end{equation*}
 and let
\begin{equation*}
w_{i}(x)=\frac{1}{2}\text{max}\left(\overline{u}_{i}(x) + \overline{v}_{i}(x) ,  0\right).
\end{equation*}
Define
\[
\Omega_{i}=\{x\in \Omega : w_{i}(x) >0\}.
\]
It is easy to show that $w_{i}\geq0,$  $w_{i} \cdot w_{j}=0$ for $ i\neq j$  and $w_{i}=\phi_{i}$ on $\partial \Omega.$ Moreover we have:
\begin{equation*}
E(w_1, \cdots, w_m)=\int_{\Omega}  \sum_{i=1}^{m}( \frac{1}{2}| \nabla w_{i}|^{2}+f_{i}w_{i}) \, dx.
\end{equation*}
 Using that  $u_{i} \cdot u_{j}=0$ and $v_{i} \cdot v_{j}=0$ for  $i \neq j,$  we obtain the following estimate:
\begin{equation}
\begin{split}\label{estimate}
\int_{\Omega}\sum_{i=1}^{m}\frac{1}{2}| \nabla w_{i}|^{2}\,dx = \sum_{i=1}^{m}\int_{\Omega_{i}} \frac{1}{8}| \nabla \overline{u}_{i}+\nabla \overline{v}_{i}|^{2} \, dx < \int_{\Omega_i}  \sum_{i=1}^{m}\frac{1}{4}(|\nabla \overline{u}_{i}|^2+|\nabla \overline{v}_{i}|^{2})\, dx\\
\leq \sum_{i=1}^{m} \int_{\Omega}  \frac{1}{4}(|\nabla{u}_{i}|^2+|\nabla {v}_{i}|^{2}) \, dx = \frac{1}{2} \left(\sum_{i=1}^{m} \int_{\Omega} \frac{1}{2}|\nabla{u}_{i}|^2 + \int_{\Omega} \sum_{i=1}^{m}\frac{1}{2}|\nabla{v}_{i}|^2\right).
 \end{split}
\end{equation}
The potential part can also be estimated as follows:
\begin{equation}\label{estimate_2}
\int_{\Omega}  \sum_{i=1}^{m}f_{i} w_{i} \, dx = \int_{\Omega}  \sum_{i=1}^{m}\frac{1}{2}f_{i} \max(\overline{u}_{i}(x) + \overline{v}_{i}(x) , 0)\, dx \leq \int_{\Omega}  \sum_{i=1}^{m}\frac{1}{2}f_{i}  (u_{i}+v_{i}) \, dx,
\end{equation}
where in the  last inequality we have used the fact that   $f_{i}$ is  positive,  ($i=1, \cdots,m$). Finally, by  adding \eqref{estimate} and \eqref{estimate_2}  we obtain
\begin{equation*}
E(w_1, \cdots, w_m) < \frac{1}{2}[E(u_1, \cdots, u_m)+ E(v_1, \cdots, v_m)] = c,
\end{equation*}
which is a contradiction. This completes the proof of Proposition.
\end{proof}

In this part we state some results that will be used in the construction of our numerical scheme. The next Lemma shows that the minimizer of the variational problem satisfies certain differential inequalities.

\begin{lemma}\label{8}
Let  $U=(u_1,\cdots, u_m)$ be a minimizer of Problem (A),  then the following holds in the sense of distributions.
\begin{equation*}
 \Delta u_{i}\geq f_{i}(x)\chi_{\{u_i>0\}}.
\end{equation*}
\end{lemma}
\begin{proof}
One needs to show that for each $i=1, \cdots m$, and  test function $\phi \in C^{\infty}_{c}(\Omega)$ the following inequality holds:
\[
\int_{\Omega}\nabla u_{i} \cdot \nabla \phi +f_{i}\chi_{\{u_i>0\}}\phi\, dx  \leq 0.
\]
For $ 0< \varepsilon<<1,$ and fixed $i$ we define new functions $(v_{1},\cdots v_{m})$ as follows:
\[v_1=u_1,\;\;v_2=u_2,\dots,v_i=(u_{i}-\varepsilon \phi)^{+},\dots,v_m=u_m.\]
It is easy to show that,
 \[v_i\cdot v_j=0,\;\; \text{whenever}\;\; i\neq j,\; \text{and}\; v_i=\phi_i\; \text{on the boundary of}\;\; \Omega.\]
Denote $V=(v_{1},\cdots v_{m}).$ We  have $E(U)\leq E(V),$ therefore
\[
0 \leq E(V)- E(U)=\int_{\Omega}|\nabla( u_{i} -\varepsilon \phi)^{+}|^{2}- |\nabla u_{i} |^{2}+f_{i}(( u_{i} -\varepsilon \phi)^{+}-u_{i})\, \phi\, dx
 \]
 \[
 = \int_{\Omega}|\nabla( u_{i} -\varepsilon \phi)^{+}|^{2}- |\nabla u_{i} |^{2}\,dx +\int_{\Omega}f_{i}\chi_{\{u_i>0\}}(( u_{i} -\varepsilon \phi)^{+}-u_{i})\, \phi\, dx
 \]
\[
    \leq-\varepsilon \int_{\Omega} \left(\nabla u_{i}\cdot \nabla \phi + f_{i}\chi_{\{u_i>0\}} \phi \,\right) dx + o(\varepsilon).
\]
Thus
\[
\int_{\Omega}\nabla u_{i} \cdot \nabla \phi +f_{i}\chi_{\{u_i>0\}}\phi\, dx  \leq 0.
\]

\end{proof}

\begin{definition}
The multiplicity of a point $x \in \overline{\Omega}$ is defined by:
\begin{equation*}
m(x)=\text{card} \left\{i:meas (\Omega_{i} \cap B(x,r))>0 \  \text {for some} \ r>0\right\},
\end{equation*}
and the interface between two densities is defined as:
\begin{equation*}
\Gamma _{i,j}=\partial \Omega_{i} \cap \partial \Omega_{j} \cap{\{x \in \Omega :m(x)=2}\}.
\end{equation*}
\end{definition}
Our numerical scheme is based on the following properties, which are straightforward to verify.
\begin{corollary}\label{9}
Assume that $x_{0} \in \Omega$ then the following holds:\\
1) If $m(x_0)=0,$ then there exists $r>0$ such that for every $i=1,\cdots m$; $u_i\equiv0$ on  $B(x_0,r).$ \\
2) If $m(x_0)=1,$ then there are $i$  and $r>0$ such that in $ B(x_0,r)$
\[
\Delta u_i=f_i, \quad \ \quad u_j\equiv0 \quad \text{for  } j\neq i.
\]
3) If $m(x_0)=2,$ then there are $i, j $ and  $r>0$ such that for every $k$  and $k\neq i,j$ we have $ u_k\equiv 0$ and
\[
\Delta(u_i-u_j) =f_{i}\chi_{\Omega_i}-   f_{j}\chi_{\Omega_j}  \text{ in  }  B(x_0,r).
\]
\end{corollary}
\subsection{Special cases of Problem (A)}
We note that the  One Phase Obstacle problem and the  Two-Phase Membrane problem  are special cases  of Problem (A)  for $m=1$  and $m=2,$ respectively. Here we  briefly explain these two problems and refer the reader  about variational inequalities to \cite{MR737005} and for
the Two-Phase Membrane problem to \cite{MR1620644}.

\begin{itemize}
\item  One Phase Obstacle problem ($m=1$).   Consider the following energy functional
\begin{equation}\label{10}
\min E(u)=\int_{\Omega}  \left(\frac{1}{2}| \nabla u|^{2}+fu\right )dx,
\end{equation}
 over the convex set $={\{u \in H^{1}(\Omega) : u \geq0, u=\phi\geq 0 \quad \text {on} \quad \partial  \Omega}\}.$
 The minimizer of \eqref{10}   satisfies the following  Euler-Lagrange equation
 \begin{equation}\label{11}
\left \{
\begin{array}{lll}
    \Delta u=f \chi_{\{u>0 \}}   & \text {in } \Omega ,\\
    u=\phi                                    &  \text   {on } \partial \Omega , \\
  u=| \nabla u |=0     &  \text {in } \Omega\backslash{\{u>0}\}.
   \end{array}
\right.
   \end{equation}

\item Two-Phase Membrane problem ($m=2$).

Let  $f_i:\Omega \rightarrow \mathbb{R},$ $i=1,2,$ be non-negative  Lipschitz continuous functions, where    $\Omega$ is  a bounded open subset of $\mathbb{R}^n$ with smooth  boundary.
Let
\[
K=\{ v \in W^{1,2}(\Omega): v-g \in W^{1,2}_{0}(\Omega) \},
\]
 where $g$ changes the sign on the boundary. Consider the functional
\begin{equation}\label{13}
I(v)=\int_{\Omega}\left(\frac{1}{2}|\nabla v|^2+f_{1} \text {max}(v,0)-f_{2}\text {min}(v,0)\right)dx,
\end{equation}
 which  is convex, weakly lower semi-continuous, and hence   attains its  infimum at some point  $u \in K$. In functional (\ref{13})  set
 \begin{align*}
u_{1} &=v^{+},  \quad    u_{2}=v^{-}, \\
g_{1} &=g^{+},  \quad    g_{2}=g^{-},
\end{align*}
 where $v^{\pm}=\max(\pm v,0)$. Then the functional $I(v)$ in \eqref{13} can be rewritten as
 \begin{equation}\label{functional}
I(u_1,\,u_2)=\int_{\Omega}\left(\frac{|\nabla u_{1}|^2}{2}+\frac{ |\nabla u_{2}|^{2}}{2}+ f_{1}u_1+ f_{2}u_{2} \right)dx,
\end{equation}
where minimization is over the set
 $$S={\{(u_1,\,u_{2})\in (H^{1}(\Omega))^{2} : u_{1} \cdot u_{2}=0,\, u_{i}\geq0, \quad u_{i}=g_{i} \quad \text {on} \quad \partial  \Omega, \, i=1,2}\}.$$
The Euler-Lagrange  equation  corresponding to the  minimizer  $u$ is given in   (\cite{MR1620644}), which is called the Two-Phase Membrane problem:
\begin{equation}\label{14}
\left \{
\begin{array}{ll}
\Delta u= f_{1} \chi_{\{u >0 \}}-f_{2} \chi_{\{u <0\}}  &  \text{in} \ \Omega, \\
  u=g\     & \text{on} \ \partial \Omega,
  \end{array}
\right.
\end{equation}
where  $ \Gamma(u) =\partial \{ x \in \Omega: u(x)>0 \}  \cup \partial\{x\in \Omega: u(x) < 0\} \cap \Omega $ is called  the \emph{free boundary.}

\end{itemize}

\subsection{Numerical approximation of Problem (A)}
In this section we present our numerical scheme, which is based on the properties in Corollary \ref{9}. It means that if $m(x)=1, x  \in B_{r},$ then  our scheme solves $ \Delta u_i=f_i$ locally. For all $x$ such that  $m(x)=2,$  the scheme solves
 \[
 \Delta(u_i-u_j) =f_{i}\chi_{\{u_i > 0\}}-   f_{j}\chi_{\{u_j   > 0\}}.
 \]
To explain our method, first  let $m=2.$ We have
\begin{equation}\label{15}
\Delta (u_{1}-u_{2})= f_1 \chi_{\{u_{1} >0\}}-f_{2} \chi_{\{u_{2} >0\}}.
\end{equation}
Equation \eqref{15} shows that $\Delta (u_{1}-u_{2})$ is bounded and therefore by classical results for elliptic PDE we have $u_{1}-u_{2} \in C^{1,\alpha}$ for $\alpha <1$.  Thus, on  the free boundary we have
\[
\nabla u_{1}= -\nabla u_{2}.
\]
For a given uniform  mesh on $\Omega\subset \mathbb{R}^2,$ we define $\overline{u}_{k}(x_{i},y_{j})$ to be the average of $u_{k}$ for all neighbor points of $(x_{i},y_{j}),$ where $k=1,2.$ Thus
\[
\overline{u}_{k}(x_{i},y_{j})=\frac{1}{4}[u_{k} (x_{i-1},y_{j})+u_{k} (x_{i+1},y_{j})+u_{k}(x_{i},y_{j-1})  +u_{k} (x_{i},y_{j+1})].
\]
 We use  the standard  finite difference discretization for equation \eqref{15}. By setting  $\triangle x=\triangle y=h,$  we arrive at
\begin{equation}\label{17}
\begin{split}
  &   \frac{1}{h^{2}}[4 \overline{u}_{1}(x_{i},y_{j}) -4u_{1} (x_{i},y_{j})]- \frac{1}{h^{2}}[4 \overline{u}_{2}(x_{i},y_{j}) -4u_{2} (x_{i},y_{j})]\\
 & = f_{1} \chi_{\{u_{1}(x_{i},y_{j}) >0\}}-f_{2}\chi_{\{u_{2}(x_{i},y_{j}) >0\}}.
  \end{split}
\end{equation}
Therefore  we obtain   $u_{1} (x_{i},y_{j})$ and  $u_{2} (x_{i},y_{j})$    from (\ref{17}) and  impose   the following  conditions
\[
u_{1} (x_{i},y_{j}) \cdot u_{2} (x_{i},y_{j}) =0  \text{ and }    u_{1}(x_i,y_{j})\geq 0, \, u_{2}(x_i,y_j)\geq 0.
  \]
Then the  iterative method for $ u_{1}$ and $ u_{2}$ will be as  follows:\\
 $\bullet $\textbf{ Initialization:}
 \begin{equation*}
 u_{1}^{(0)}(x_{i},y_{j})=
\left \{
\begin{array}{ll}
  0    &    (x_{i},y_{j}) \in \Omega^{\circ},   \\
  \phi_{1}(x_{i},y_{j})   &   (x_{i},y_{j}) \in \partial \Omega.
\end{array}
\right.
\end{equation*}
\begin{equation*}
u_{2}^{(0)}(x_{i},y_{j})=
\left \{
\begin{array}{ll}
  0                    &   (x_{i},y_{j}) \in \Omega^{\circ},     \\
  \phi_{2}(x_{i},y_{j})  &     (x_{i},y_{j}) \in \partial \Omega,
\end{array}
\right.
\end{equation*}
where $\Omega^{\circ}$ stands for the interior points of the  domain $\Omega$.\\
$\bullet $ \textbf{{Step $k+1 $, $k\geq 0:$}}

We iterate over all interior points by setting
\begin{equation*}
\begin{cases}
u_1^{(k+1)} (x_{i},y_{j}) = \max\left(\frac{- f_1(x_i,y_j) h^2}{4}+\overline{u}_{1}^{(k)}(x_{i},y_{j})-\overline{u}_{2}^{(k)}(x_{i},y_{j}),\quad 0\right),\\

u_2^{(k+1)} (x_{i},y_{j}) =\max\left( \frac{- f_{2}(x_i,y_j) h^2}{4}+\overline{u}_{2}^{(k)}(x_{i},y_{j})-\overline{u}_{1}^{(k)}(x_{i},y_{j}), \quad 0\right).
\end{cases}
\end{equation*}

Note that if $m=1,$ then the above method can be modified. The convergence of the method in this case is given in \cite{MR737005}. Suppose there is a grid on the domain $\Omega,$  then our  method for the case of an arbitrary $m$  densities can be formulated as follows:
 \begin{itemize}
 \item \textbf{ Initialization:}
$\text{For}\quad    l=1,\cdots, m, \;\; \text{set}$
\begin{equation*}
u_{l}^{0}(x_{i},y_{j})=
\left \{
\begin{array}{ll}
  0  &    (x_{i},y_{j})   \in \Omega^{\circ},  \\
  \phi_{l}(x_{i},y_{j})  &   (x_{i},y_{j}) \in \partial \Omega.
\end{array}
\right.
\end{equation*}
\item  \textbf{{Step $k+1 $, $k\geq 0$:}}
For $l=1,\cdots, m,$ we iterate  for all  interior points
  \begin{equation}\label{I}
  u_l^{(k+1)} (x_{i},y_{j}) =\max \left(\frac{-f_{l}h^{2}}{4}+\overline{u}_{l}^{(k)}(x_{i},y_{j}) - \sum_{p \neq l}  \overline{u}_{p}^{(k)}(x_{i},y_{j}), \,  0\right).
  \end{equation}
  \end{itemize}
 \begin{remark}
Note that this iterative method is slow, since  the information  propagates  from the boundary into the domain. One interesting question  is, how can the idea of multi- grid method  be applied?
\end{remark}
\begin{lemma}\label{lemma}
The iterative method \eqref{I} satisfies
\[u_l^{(k)}(x_{i},y_{j})\cdot u_q^{(k)}(x_{i},y_{j})=0,\]
for all $k\in\mathbb{N}$ and $q,l\in\{1,2,\dots,m\},where\;\; q\neq l.$
\end{lemma}
\begin{proof}
Observe that from \eqref{I} it follows that  \[u_l^{(k)}(x_{i},y_{j})\geq 0,\] for all $k\in\mathbb{N}$ and $l\in\{1,2,\dots,m\}.$ Assume $u_l^{(k)}(x_{i},y_{j})> 0$ then by \eqref{I} we have
\[u_l^{(k)} (x_{i},y_{j}) =\frac{-f_{l}h^{2}}{4}+\overline{u}_{l}^{(k-1)}(x_{i},y_{j}) - \sum_{p \neq l}  \overline{u}_{p}^{(k-1)}(x_{i},y_{j}).\]
This shows that
\[\overline{u}_{l}^{(k-1)}(x_{i},y_{j})>\sum_{p \neq l}\overline{u}_{p}^{(k-1)}(x_{i},y_{j})+\frac{f_{l}h^{2}}{4}\geq\overline{u}_{q}^{(k-1)}(x_{i},y_{j}).\]
Thus
\[\overline{u}_{q}^{(k-1)}(x_{i},y_{j})<\overline{u}_{l}^{(k-1)}(x_{i},y_{j})\leq\frac{f_{q}h^{2}}{4}+\sum_{p \neq q}  \overline{u}_{p}^{(k-1)}(x_{i},y_{j}),\]
and after rearranging above inequalities we arrive at
\begin{equation}\label{ineq_lemma}
\frac{-f_{q}h^{2}}{4}+\overline{u}_{q}^{(k-1)}(x_{i},y_{j})-\sum_{p \neq q}  \overline{u}_{p}^{(k-1)}(x_{i},y_{j})<0.
\end{equation}
In light of \eqref{I}  and \eqref{ineq_lemma} we derive
\[u_q^{(k)} (x_{i},y_{j}) =\max \left(\frac{-f_{q}h^{2}}{4}+\overline{u}_{q}^{(k-1)}(x_{i},y_{j}) - \sum_{p \neq q}  \overline{u}_{q}^{(k-1)}(x_{i},y_{j}), \,  0\right)=0.\]
Thus
\[u_l^{(k)}(x_{i},y_{j})\cdot u_q^{(k)}(x_{i},y_{j})=0.\]
\end{proof}

In order to see the consistency of the method to the problem \eqref{1}, we will consider the finite difference scheme of our method \eqref{I}. The scheme apparently will be the following discrete nonlinear system :
\begin{equation}\label{II}
\small
\begin{cases}
u_l (x_{i},y_{j}) =\max \left(\frac{-f_{l}h^{2}}{4}+\overline{u}_{l}(x_{i},y_{j}) - \sum_{p \neq l} \overline{u}_{p}(x_{i},y_{j}),\,0\right)&(x_{i},y_{j})\in \Omega^{\circ},\\
u_{l}(x_{i},y_{j})=\phi_{l}(x_{i},y_{j})&(x_{i},y_{j}) \in \partial \Omega,
\end{cases}
\end{equation}
where $l\in\{1,2,\dots,m\}.$

We want to show the consistency of the scheme \eqref{II} to the discussed properties in Corollary \ref{9}.
First of all the disjoint property of the components follows directly from  Lemma \ref{lemma}. Suppose  $u_{l}(x_{i},y_{j})>0,$ and together with this $\overline{u}_{p}(x_{i},y_{j})=0$, for all $p\neq l.$ This will imply that \[u_l (x_{i},y_{j})=\frac{-f_{l}h^{2}}{4}+\overline{u}_{l}(x_{i},y_{j}) - \sum_{p \neq l} \overline{u}_{p}(x_{i},y_{j})=\frac{-f_{l}h^{2}}{4}+\overline{u}_{l}(x_{i},y_{j}),\]
and hence
\begin{equation}\label{discret_laplace}
\frac{1}{h^2}(4\overline{u}_{l}(x_{i},y_{j})-4u_{l} (x_{i},y_{j}))=f_l(x_{i},y_{j}).
\end{equation}
But  equation \eqref{discret_laplace} is just a discrete scheme of the Poisson equation
\[\Delta u_l=f_l.\]
Hence, if in the  discrete sense $u_l(x,y)>0,$ then we have $\Delta u_l=f_l.$
If we are locally on the free boundary of two components, say $u_l$ and $u_q,$ then in the scheme \eqref{II} we have the following situation:
\[u_l(x_i,y_j)=u_q(x_i,y_j)=\overline{u}_{p}(x_{i},y_{j})=0,\;\;\text{where}\;\; p\neq l\;\; \text{and} \;\; p\neq q.\]
According to the scheme \eqref{II} we have
\[0=\max\left(\overline{u}_{l}(x_{i},y_{j})-\overline{u}_{q}(x_{i},y_{j})-\frac{f_{l}h^{2}}{4},0 \right),\]
and
\[0=\max\left(\overline{u}_{q}(x_{i},y_{j})-\overline{u}_{l}(x_{i},y_{j})-\frac{f_{q}h^{2}}{4},0 \right).\]
Therefore
\[-f_{q}(x_{i},y_{j})\leq\frac{4}{h^2}(\overline{u}_{l}(x_{i},y_{j})-\overline{u}_{q}(x_{i},y_{j}))\leq f_{l}(x_{i},y_{j}),\]
and taking into account $u_l(x_i,y_j)=u_q(x_i,y_j)=0,$ we obtain
\[-f_q\leq\Delta_h(u_l-u_q)\leq f_l,\]
at  $(x_{i},y_{j}).$
Combining all results we see the consistency with Corollary \ref{9}.

Here we give a proof of the convergence of our method to the discretized problem, in the  case $m=2$ and $f_{i}=0.$ We consider the following non-linear finite difference method
\begin{equation}\label{f}
\begin{cases}
u_{1}^{k+1}(x_{i},y_{j})=\max(\overline{u}_{1}^k-\overline{u}_{2}^k,0),\\
u_{2}^{k+1}(x_{i},y_{j})=\max(\overline{u}_{2}^k-\overline{u}_{1}^k,0).
\end {cases}
\end{equation}
Note that (\ref{f}) can be written as:
\begin{equation}
\begin{cases}
u_{1}^{k+1}(x_{i},y_{j})=\max(\overline{u}_{1}^k-\overline{u}_{2}^k,0)=\frac{1}{2}\left(\overline{u}_{1}^k-\overline{u}_{2}^k+|\overline{u}_{1}^k-\overline{u}_{2}^k|\right),\\
u_{2}^{k+1}(x_{i},y_{j})=\max(\overline{u}_{2}^k-\overline{u}_{1}^k,0)=\frac{1}{2}\left(\overline{u}_{2}^k-\overline{u}_{1}^k+|\overline{u}_{2}^k-\overline{u}_{1}^k|\right).
\end {cases}
\end{equation}
By subtracting the first equation from the second, we obtain
\begin{equation}\label{lap_u1-u2}
u_{1}^{k+1}(x_{i},y_{j})-u_{2}^{k+1}(x_{i},y_{j})=\overline{u}_{1}^k-\overline{u}_{2}^k,
\end{equation}
which is  a classical finite difference scheme of $\Delta(u_{1}(x)-u_{2}(x))=0$. It is noteworthy that \eqref{lap_u1-u2} follows from the last part in Corollary \ref{9}. This gives that  we have convergence of
\begin{equation}\label{ff}
u_{1}^{k}(x_{i},y_{j})-u_{2}^{k}(x_{i},y_{j})
\end{equation}
 at every point $(x_{i},y_{j})$,  when $k\rightarrow\infty$.  Recalling that by Lemma \ref{lemma}
   \[u_{1}^{k}(x_{i},y_{j})\cdot u_{2}^{k}(x_{i},y_{j})=0,\]
 for  every $k>0,$   we can write  the following identity for all $k$,
\begin{equation}\label{sag}
(u_{1}^{k}(x_{i},y_{j})-u_{2}^{k}(x_{i},y_{j}))^2=(u_{1}^{k}(x_{i},y_{j})+u_{2}^{k}(x_{i},y_{j}))^2.
\end{equation}
  Therefore  convergence of  $u_{1}^{k}(x_{i},y_{j})-u_{2}^{k}(x_{i},y_{j})$ at every point $(x_{i},y_{j})$ will imply the convergence of
\[(u_{1}^{k}(x_{i},y_{j})-u_{2}^{k}(x_{i},y_{j}))^2,\]
 at every point as well. Hence, by  \eqref{sag} the sequence
   \begin{equation*}
(u_{1}^{k}(x_{i}, y_{j})+u_{2}^{k}(x_{i},y_{j}))^2
\end{equation*}
 converges at every point $(x_{i}, y_{j})$. Note that  $u_{1}^{k}(x_{i},y_{j})$ and $u_{2}^{k}(x_{i},y_{j})$ are positive, which implies the convergence of
\[
u_{1}^{k}(x_{i},y_{j})+u_{2}^{k}(x_{i},y_{j}).
\]
  Finally, convergence of $u_{1}^{k}(x_{i},y_{j})-u_{2}^{k}(x_{i},y_{j})$ and $u_{1}^{k}(x_{i},y_{j})+u_{2}^{k}(x_{i},y_{j})$   will imply the convergence of $u_{1}^{k}(x_{i},y_{j})$ and $u_{2}^{k}(x_{i},y_{j})$ at every nodal point $(x_{i},y_{j})$.  This completes the proof.

\section{ Theoretical results of Problem (B)}
\renewcommand{\theequation}{3.\arabic{equation}}
\setcounter{equation}{0}

In this section we present results that have been proved  for Problem (B), for the case of two-species in dimension two.

Consider the following system:
\begin{equation}\label{19}
\left \{
\begin{array}{llllll}
 u_{t}-d_{1} \Delta u = \lambda u(1 - u)-\frac{1}{\varepsilon} uv^{2}   & \text{ in } \Omega \times (0, \infty ),\\
  v_{t}- d_{2} \Delta v = \lambda v(1 - v)-\frac{1}{\varepsilon} u^{2}v  & \text{ in } \Omega \times (0, \infty ),\\
  u(x, y, t) = \phi(x, y, t)   & \text{ on } \partial \Omega \times (0, \infty ),\\
  v(x, y, t) = \psi(x, y, t)   & \text{ on } \partial \Omega \times (0, \infty ),\\
   u(x, y, 0) = u_{0}(x, y)    & \text{ in } \Omega,\\
   v(x, y, 0) = v_{0}(x, y)    & \text{ in } \Omega.\\
\end{array}
\right.
\end{equation}
This problem has been studied in \cite{MR2363653, MR1293103, MR2459673,MR2417905}, where the references of some physical background  involving  cubic coupling  is given. The system \eqref{19},  for steady boundary data  admits a Lyapunov energy.
Assume that the initial conditions $ u_{0}(x, y)$ and $v_{0}(x, y)$   have disjoint  supports   and
 \[
   0 \leq u_{0}(x, y), v_{0}(x, y)\leq 1.
 \]
 We  also assume that the boundary conditions are positive with disjoint support.  The following Theorem has been proved in \cite{MR2417905}.
\begin{theorem}\label{25}
There exist two functions $u(x, y), v(x, y) \in H^{1}(\Omega)\cap L^{\infty}(\Omega)$ such that
\[
(u_{\varepsilon_{m}}(t_{m}), v_{\varepsilon_{m}}(t_{m}))\rightarrow (u, v) \text{ in  } L^{p}(\Omega)\times L^{p}(\Omega) \text{ for any } p\geq 2,
\]
as $ \varepsilon \rightarrow 0 $ and $ t\rightarrow \infty,$ where $ 0 \leq  u, v \leq1$ and $ u \cdot v =0$ in $\Omega.$ Moreover,
\[
-d_{1} \Delta u \leq \lambda u(1-u), \quad   -d_{2} \Delta v \leq \lambda v(1-v),
\]
and  $u\lfloor_{\partial \Omega}=\phi$, $v \lfloor_{\partial \Omega}=\psi.$
\end{theorem}
Next, consider the following system:
\begin{equation}\label{26}
\left \{
\begin{array}{llllll}
 u_{t}- d_{1} \Delta u = \lambda u(1 - u)-\frac{1}{\varepsilon} uv   & \text{ in } \Omega \times (0, \infty ),\\
  v_{t}- d_{2} \Delta v = \lambda v(1 - v)-\frac{1}{\varepsilon} uv  & \text{ in } \Omega \times (0, \infty ),\\
  u(x, y, t) = \phi(x, y, t)  & \text{ on } \partial \Omega \times (0, \infty ),\\
  v(x, y, t) = \psi(x, y, t) & \text{ on } \partial \Omega \times (0, \infty ),\\
   u(x, y, 0) = u_{0}(x, y)  & \text{ in } \Omega,\\
   v(x, y, 0) = v_{0}(x, y)  & \text{ in } \Omega.\\
\end{array}
\right.
\end{equation}

It has been shown in \cite{MR2079274}    that for any $ T> 0$ as $\varepsilon$ tends to zero there exists a sequence of solutions
$(u_{\varepsilon},v_{\varepsilon})$   to the system  (\ref{26})    converging  in $ L^{2}(\Omega \times (0, T))$   to a bounded segregated state $(u, v),$ such that $w = u - v$  solves the  limiting free boundary problem \eqref{30}, which  shows the spatial segregation phenomena  on finite time intervals.
 \begin{theorem}\label{28}\cite{MR2079274}
 Let $T>0$. Then there exists a sequence $\varepsilon _{m}$  and $u, v \in L^{\infty}$ with
 \[
 ( u_{\varepsilon_{m}}, v_{\varepsilon_{m}}) \rightarrow (u, v) \quad \text{ in } L^{2}(\Omega \times (0, T)) \times  L^{2}(\Omega \times (0, T)),
 \]
 as $\varepsilon\rightarrow 0$, where $ 0 \leq u, v\leq 1$ and $ u\cdot v=0$  in $\Omega.$ Moreover, $w=u-v$ is the unique  weak solution to the following free boundary problem:
 \begin{equation}\label{30}
\left \{
\begin{array}{lll}
 w_{t}-\Delta  D (w) = \lambda w(1 - |w|)   & \text{ in } \Omega \times (0, \infty ),\\
  D(w(x, y, t)) = d_{1}\phi(x, y, t)-d_{2}\psi(x, y, t)  & \text{ on } \partial \Omega \times (0, \infty ),\\
  w(x, y, 0) = u_{0}(x, y)-v_{0}(x, y)  & \text{ in } \Omega,\\
   \end{array}
\right.
\end{equation}
 where
 \begin{equation}\label{32}
 D(\sigma)=
\left \{
\begin{array}{ll}
 d_{1} \sigma   & \sigma\geq 0,\\
  d_{2} \sigma   & \sigma  < 0.
   \end{array}
\right.
\end{equation}
 \end{theorem}
The  cases of time-dependent boundary conditions and possibly different diffusion coefficients has been discussed in \cite{MR2079274}. In the case of equal diffusion coefficients $d_1 = d_2 $ and stationary boundary conditions, Crooks, Dancer and Hilhorst  studied the long-term segregation for large interactions (see \cite{MR2300320}). They reduced  the system to a single parabolic equation, whose solution have $\varepsilon$-independent uniform bounds. This system  does not admit a natural Lyapunov functional and therefore a direct analysis is not possible  for long term behavior.

\subsection{  Numerical approximation of Problem (B) }

We present a numerical scheme for elliptic system in Problem (B) as $\varepsilon \rightarrow 0$. To explain the method,  we assume that there exist  two components. The Theorem \ref{25} states that
\[
d_{2} \Delta v  -d_{1} \Delta u =\lambda u(1-u)\chi_{\{u > 0\}} - \lambda v(1-v)\chi_{\{v > 0\}}.
\]
This equation  is  solved numerically by employing second order, centered, finite
difference scheme on the given  grid  i.e,
\begin{equation}\label{34}
    -\frac{d_{1}}{h^{2}}[4 \overline{u}(x_{i},y_{j}) -4u (x_{i},y_{j})]+\frac{d_{2}}{h^{2}}[4 \overline{v}(x_{i},y_{j}) -4v (x_{i},y_{j})]=
  \end{equation}
\begin{equation*}
 \lambda u(x_{i},y_{j})(1-u(x_{i},y_{j}))\chi_{\{u(x_{i},y_{j})   > 0\}} - \lambda v (x_{i},y_{j})(1-v (x_{i},y_{j}))\chi_{\{v(x_{i},y_{j})   > 0\}}.
\end{equation*}
It is easy to see that the equation  (\ref{34}) is a   quadratic equation with respect to   $u (x_{i},y_{j})$ and  $v (x_{i},y_{j}).$  Using the same  approach as in Section 2.2, if   $u (x_{i},y_{j})>0,$ then we set $v (x_{i},y_{j})=0$ and vice versa. Set $ 4\alpha=\lambda h^{2},$ then from  equation (\ref{34}) we  have the following iterative formulas:
{\normalsize
\[
u^{(k+1)} (x_{i},y_{j})=\max \left(\frac{2( d_{1}\overline{u}^{(k)}(x_{i},y_{j})-d_{2}\overline{v}^{(k)}(x_{i},y_{j}))}{d_{1}-\alpha +\sqrt{(d_{1}-\alpha)^{2}+4\alpha(d_{1}\overline{u}^{(k)}(x_{i},y_{j})-d_{2}\overline{v}^{(k)}(x_{i},y_{j}))}},0 \right),
\]}
and
{\normalsize
\[
 v^{(k+1)}(x_{i},y_{j})=\max \left(\frac{2(d_{2}\overline{v}^{(k)}(x_{i},y_{j})-d_{1}\overline{u}^{(k)}(x_{i},y_{j}))}{d_{2}-\alpha +\sqrt{(d_{2}-\alpha)^{2}+4\alpha(d_{2}\overline{v}^{(k)}(x_{i},y_{j})-d_{1}\overline{u}^{(k)}(x_{i},y_{j}))}},0 \right).
\]
}

This approach can be extended for $m$ components as well. The idea is just we take the difference between the i-th equation of the system and the sum of all other equations. After that we use the same disjointness approach, by setting  $u_i(x_s,y_r)>0$ and $u_j(x_s,y_r)=0$ for all $i\neq j,$ on the grid point $(x_s,y_r).$ This will lead us to the quadratic equation  w.r.t $u_i(x_s,y_r)$ as above.
Thus according to the same arguments as above for $m$ components we obtain the following iterative method:
For all $l = 1,\dots,m,$
\begin{equation}\label{iterate_method}
u_l^{(k+1)}(x_{i},y_{j})=\max \left(\frac{ 2\overline{w_l}^{(k)}(x_{i},y_{j}) }{d_{l}-\alpha +\sqrt{(d_{l}-\alpha)^{2}+4\alpha\overline{w_l}^{(k)}(x_{i},y_{j})}},0 \right),
\end{equation}
where
\[\overline{w_l}^{(k)}(x_{i},y_{j})=d_{l}\overline{u_l}^{(k)}(x_{i},y_{j})-\sum _{p\neq l}d_{p}\overline{u_p}^{(k)}(x_{i},y_{j}).\]
Again using the same approach as in Lemma \ref{lemma}, one can prove the same result for this method as well.
\begin{lemma}\label{lemma_2}
If$\,\,\,\underset{l}{\min }\,d_l>\alpha,$ then the iterative method \eqref{iterate_method} satisfies
\[u_l^{(k)}(x_{i},y_{j})\cdot u_q^{(k)}(x_{i},y_{j})=0,\]
for all $k\in\mathbb{N}$ and $q,l\in\{1,2,\dots,m\},where\;\; q\neq l.$
\end{lemma}
\subsection{ Parabolic case  }
 In the case, when coupling term is $\frac{uv}{\varepsilon},$ the Theorem \ref{14} states  that $w = u - v$  solves the  limiting
free boundary problem in Theorem  \ref{30}, which  shows the spatial segregation phenomena  on finite time intervals. In order to solve the problem \eqref{34} the second-order, implicit, Crank-Nicolson method is applied.

\begin{equation}\label{36}
\begin{split}
& \frac{w^{n+1}(x_{i},y_{j})-w^{n}(x_{i},y_{j})}{dt} - \frac{1}{2} (\Delta D w|^{n+1}_{(x_{i},y_{j})}+ \Delta D w|^{n}_{(x_{i},y_{j})} ) \\
&    = \frac{\lambda}{2}  [w^{n+1}(1-w^{n+1})+w^{n}(1-w^{n})].
  \end{split}
\end{equation}

 In this case  we can obtain an  iterative formula for $w^{n+1}(x_{i},y_{j})$  as a function of  $w^{n}(x_{i},y_{j}),$ $\overline{w}^{n}(x_{i},y_{j})$ and  $\overline{w}^{n+1}(x_{i},y_{j}).$

\section{  Numerical Examples }
\renewcommand{\theequation}{4.\arabic{equation}}
\setcounter{equation}{0}

In this section we present  different examples  of Problem  (A) and Problem (B).    We consider the following minimization problem
\begin{equation}\label{last}
I=\int_{\Omega}  \sum_{i=1}^{m} \left( \frac{1}{2}| \nabla u_{i}|^{2}+f_{i}u_{i}  \right) dx,
\end{equation}
  over the set $S={\{(u_1,\dots,u_{m})\in (H^{1}   (\Omega))^{m} :u_{i}\geq0, u_{i} \cdot u_{j}=0, u_{i}=\phi_{i} \quad \text {on} \quad \partial  \Omega}\}.$
Examples 1, 2 and 3  show the numerical approximations of  Problem (A)  for different values $m$ and different $\Omega.$

\begin{example}
Figure \ref{fig:E_2} shows the solution of Problem \eqref{last} in the case of $n=1, m=2$. We choose $f_1=2+sin(x), f_2=1+x^{2}.$ The equation for $u_{1}-u_{2}$ is as follows:

\begin{equation}\label{farid2}
\left \{
\begin{array}{ll}
(u_{1} -u_{2})''=(2+sin x)\chi_{\{u_{1}>0\}} -(1+x^{2})\chi_{\{u_{2}>0\}}, & x \in [-2, 2] \\
u_{1}(-2)=1, \, u_{2}(2)=1.\\
\end{array}
\right.
\end{equation}

\begin{figure}[!htbp]
{\includegraphics[width=.5 \columnwidth]{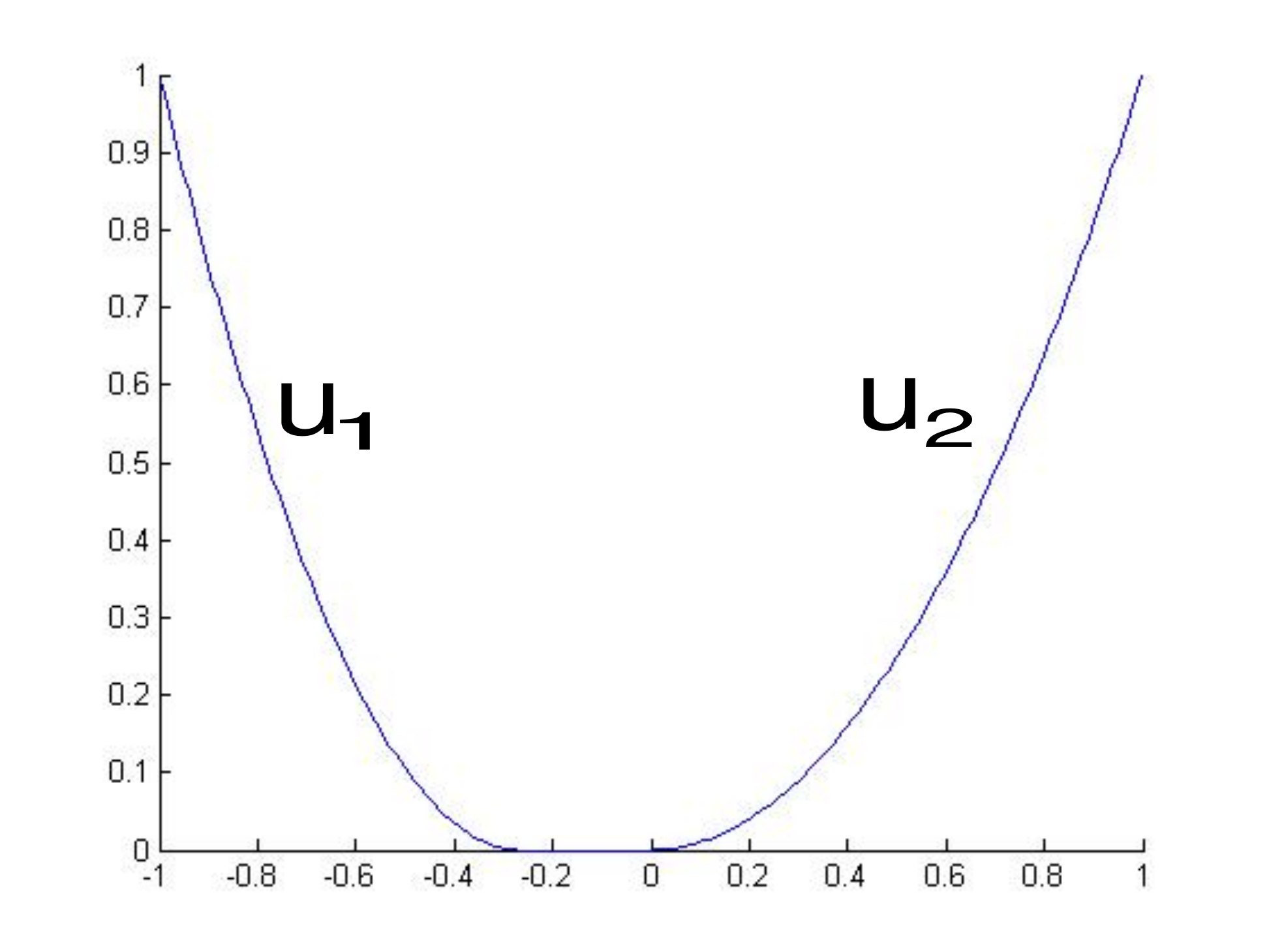}}
\centering
  \caption{The plot of $u_{1}+u_{2}.$}
  \label{fig:E_2}
 \end{figure}
\end{example}
\begin{example}
Consider Problem (A) with $m=3$ and $f_1=f_2=1, f_3=\frac{1}{4}.$  The free boundary is shown in  Figure  \ref{fig:E_3}. The boundary value  $g$ is given by
\begin{equation*}
g(x,y)=
\left \{
\begin{array}{lll}
4-x^{2} &   -2 \leq x \leq +2 \  {\&} \  y=-2,\\
4-y^{2} & -2 \leq y \leq +2 \  {\&} \  x=-2,\\
\frac{4-x^{2}}{2} &  -2 \leq x \leq +2 \  {\&} \  y=-x.\\
\end{array}
\right.
\hspace{0.1in}
\end{equation*}
\end{example}

\begin{figure}[!htbp]
 \begin{center}
  \subfloat[The free boundaries]{\includegraphics[width=.4\textwidth]{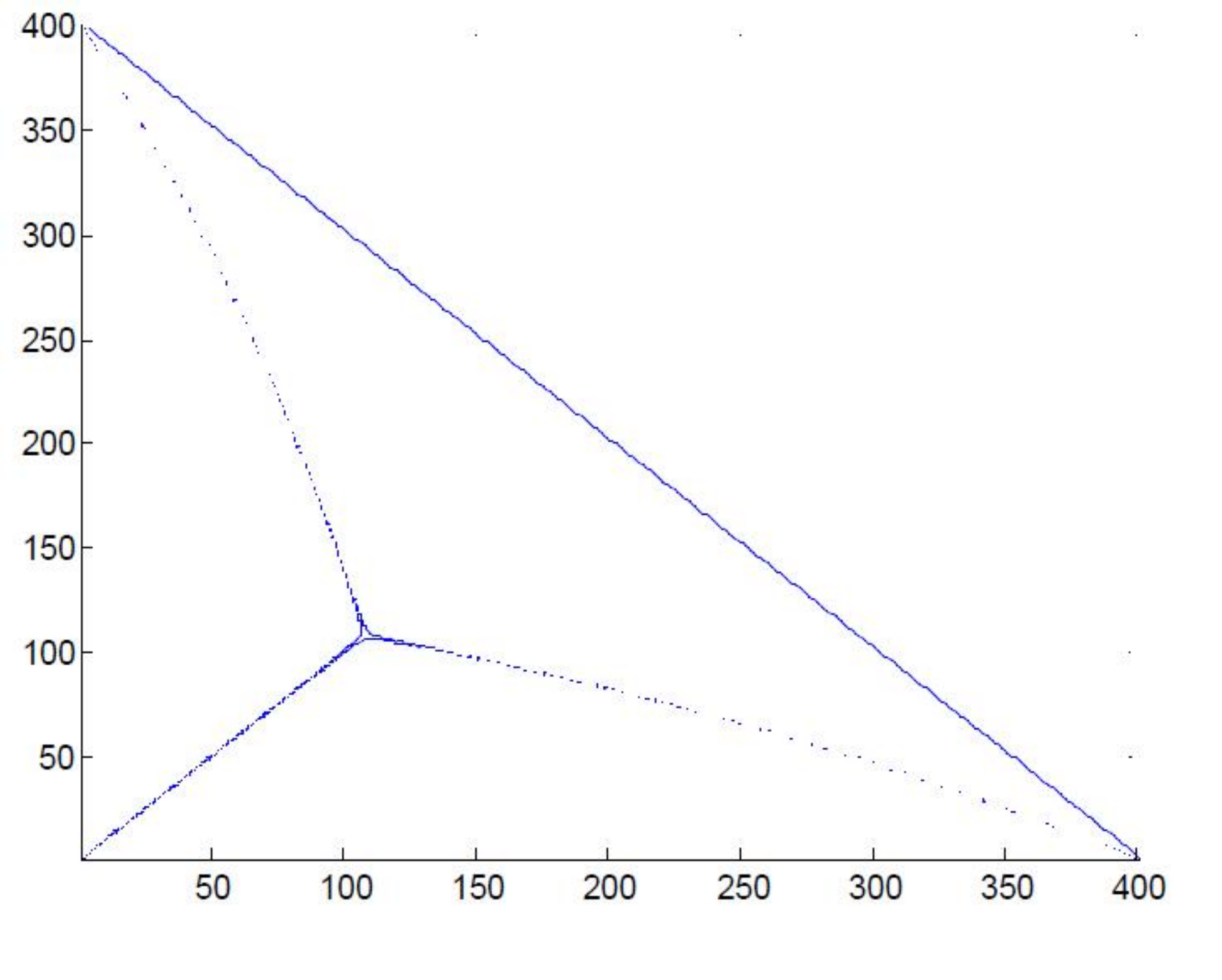}}
  \subfloat[$u_1+u_2+u_3$]{\includegraphics[width=.47\textwidth]{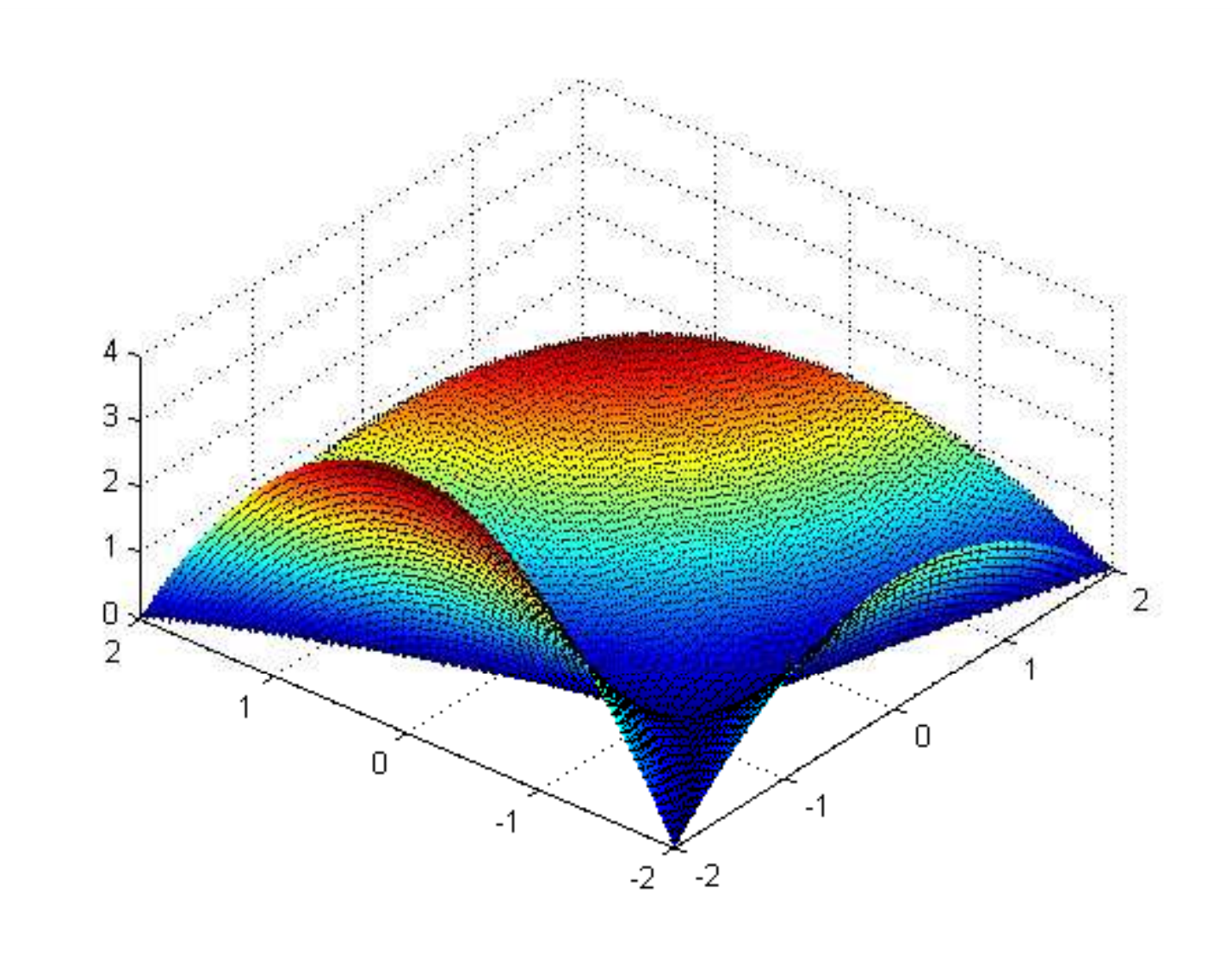}}
  \end{center}
\caption{\small{The left picture shows the free boundaries of solutions. The right picture shows the surface of $u_1+u_2+u_3.$} }\label{fig:E_3}
\end{figure}

\begin{example}\label{farid20}
Let $\Omega =[-1,1]\times [-1, 1]$ and $m=4, f_1=8;f_2=6;f_3=2;f_4=1.$
The boundary values $\phi_{i},$ (i=1,2,3,4) are given as follows:
\begin{equation*}
\phi_{1}=
\left \{
\begin{array}{lr}
1-x^{2} &   x\in[-1, 1] \  {\&} \  y=1 ,\\
0 & \ \ \text {elsewhere.}
\end{array}
\right.
\hspace{0.1in}
\phi_{2}=
\left \{
\begin{array}{lr}
1-y^{2} & y \in [-1, 1] \  {\&} \  x=1 ,\\
0 & \ \ \text {elsewhere.}
\end{array}
\right.
\end{equation*}

\begin{equation*}
\phi_{3}=
\left \{
\begin{array}{lr}
1-x^{2} &  x\in [-1 ,1] \  {\&} \  y=-1 ,\\
0 & \ \ \text {elsewhere.}
\end{array}
\right.
\hspace{0.1in}
\phi_{4}=
\left \{
\begin{array}{lr}
1-y^{2} &   y \in [-1 , 1] \  {\&} \  x=-1 ,\\
0 & \ \ \text {elsewhere.}
\end{array}
\right.
\end{equation*}
\end{example}

\begin{figure}\label{45}
  \centering
 \subfloat[contour]{\label{fig:gull}\includegraphics[width=0.40\textwidth]{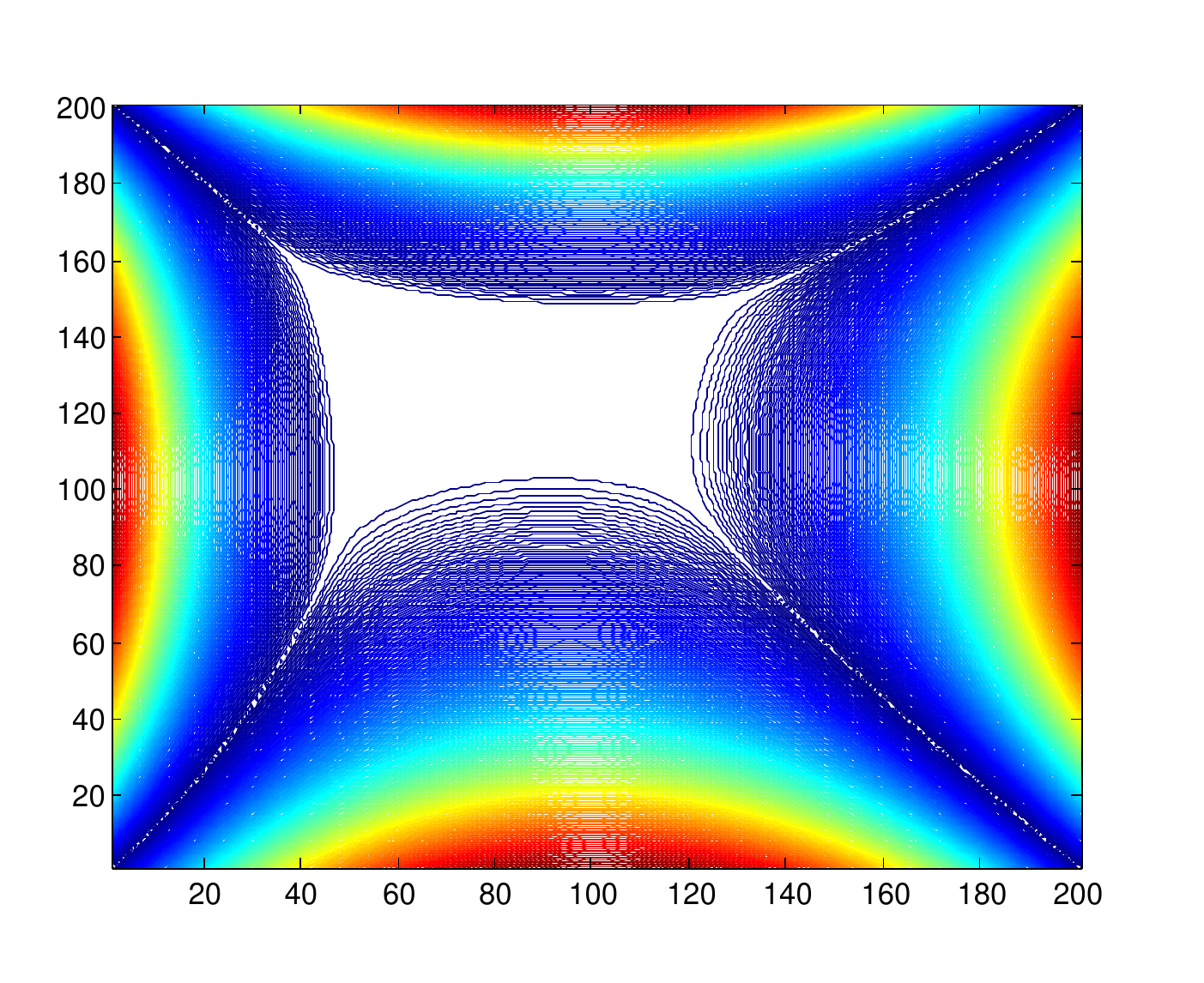}}
  \subfloat[$u_1+u_2+u_3+u_4$ ]{\label{fig:tiger}\includegraphics[width=0.50\textwidth]{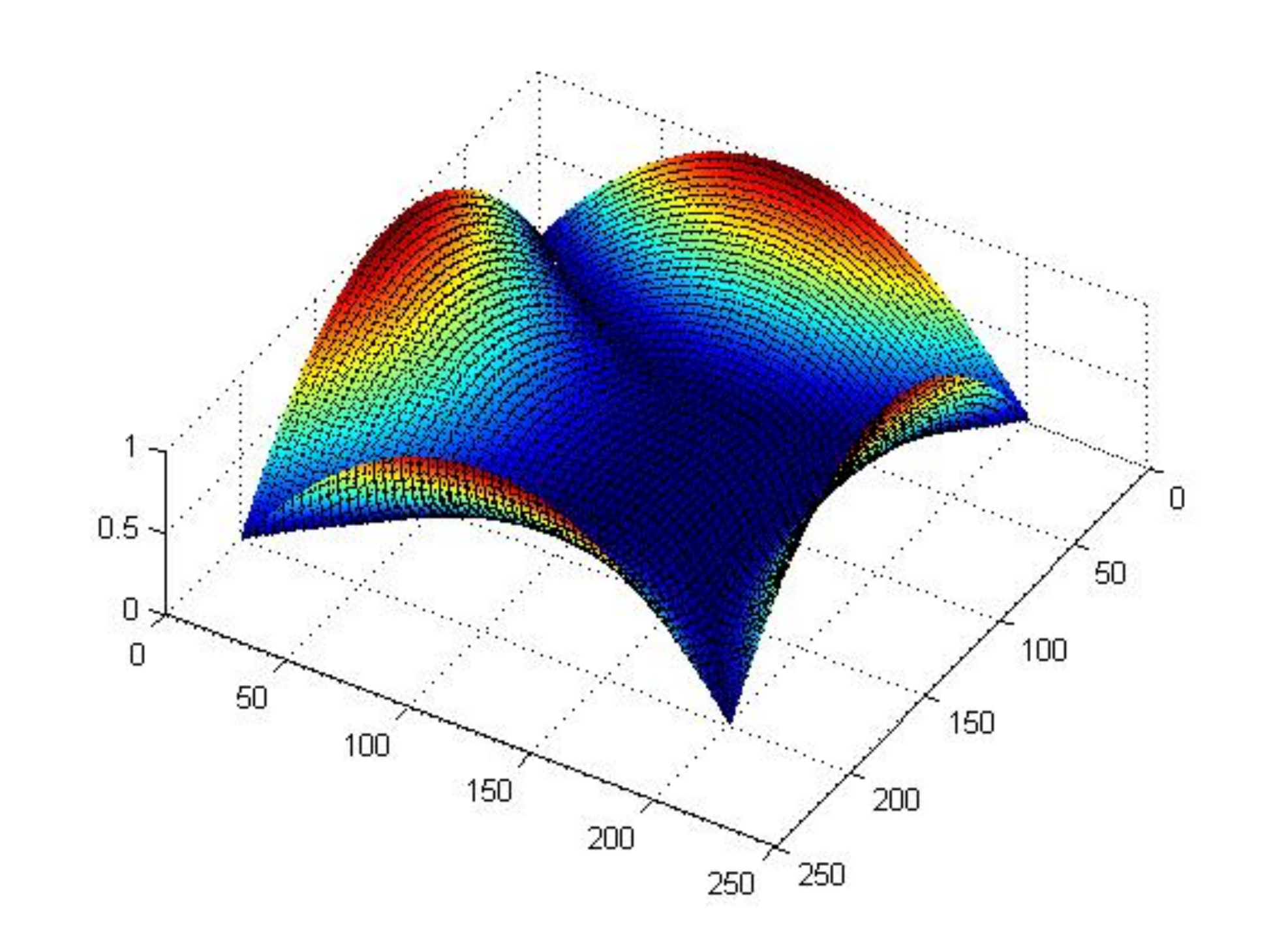}}
  \caption{\small{The left picture shows the contours of solutions  and zero set. The right picture shows the surface of $u_1+u_2+u_3+u_4.$} }
  \label{fig:animals}
\end{figure}

\begin{example}
Let $\Omega$ be as in  previous example and $m=4, f_{1}=0, f_{2}=|x^{2}-y^{2}|, f_{3}=8, f_{4}=|x+y|.$ The boundary conditions  $\phi_{i},$ (i=1,2,3,4)  are the same as in Example \ref{farid20}. The interfaces are shown in Figure 4.
\begin{figure}[!htbp]
{\includegraphics[width=.45 \columnwidth]{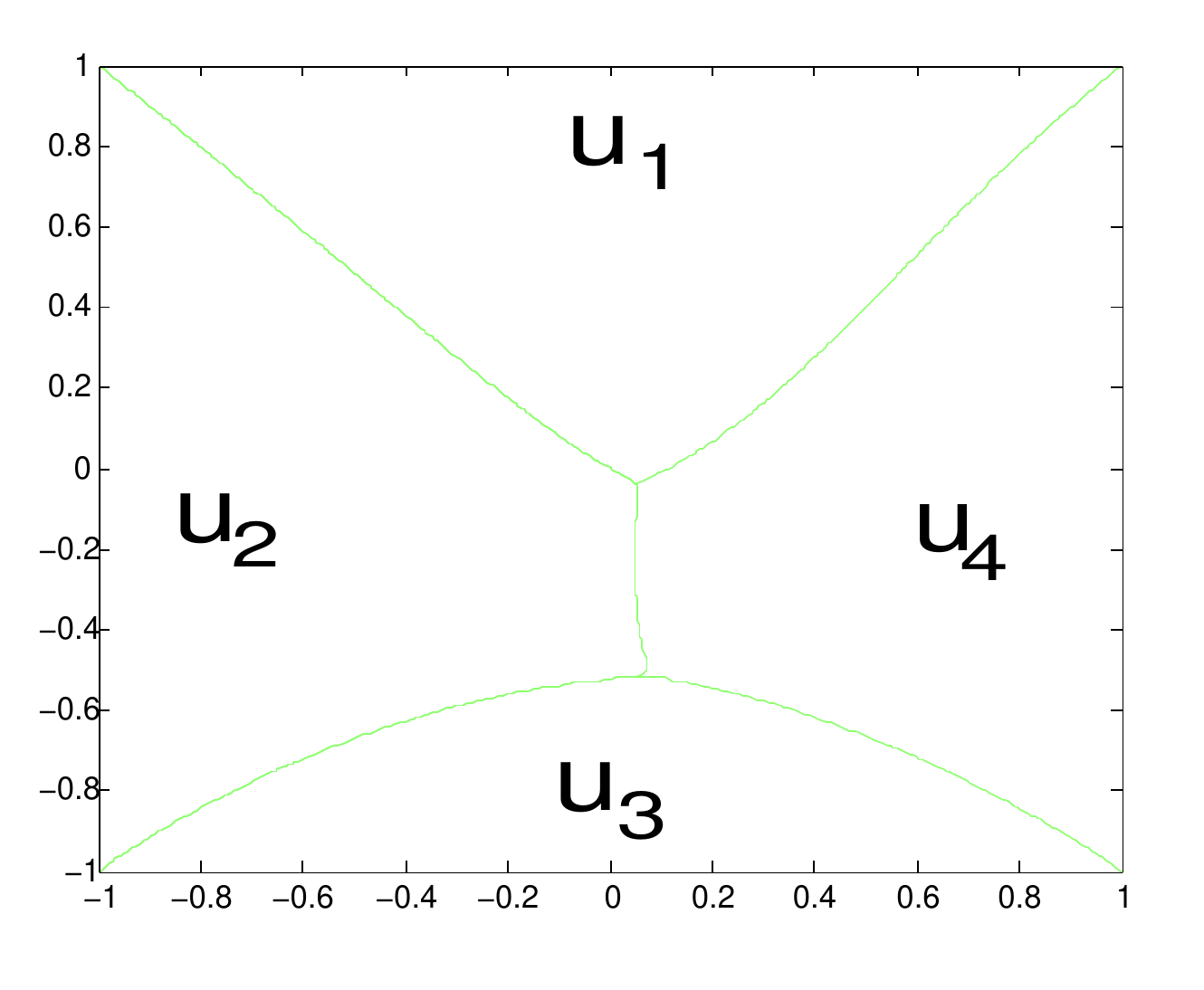}}
\centering
  \caption{\small{The picture shows the free boundaries  of solutions  of $u_1, u_2, u_3$ and $u_4.$} }
  \label{fig:E_2}
 \end{figure}
\end{example}

Now consider the following  system of $m$ differential equations for $i=1,\cdots,m,$  as $\varepsilon \rightarrow 0,$
\begin{equation}\label{farid}
\left \{
\begin{array}{lll}
-d_i\Delta u_{i}= \lambda u_{i}(1-u_{i})- \frac{1}{\varepsilon} u_{i}(x)\sum_{j\neq i}^{m} u_{j}^{2}(x)   & \text{ in } \Omega, \\
  u_{i} \geq 0      & \text{ in } \Omega,\\
    u_{i}(x) =\phi_{i}(x)  &  \text{ on }   \partial\Omega. \\
\end{array}
\right.
\end{equation}
\[
\]
\[
\]
\begin{example}
  Let $ \Omega =[0,1]\times [0, 1], m=2,\lambda=1,d_1=1.5,d_2=1.$   The steady  boundary values for $u(x,y,t), v(x,y,t)$  are defined by
\begin{equation*}
\phi(x,0,t)=
\left \{
\begin{array}{ll}
0.5-2.5 x   &  0 \leq x  \leq 0.2,\\
          0 & 0.2 \leq x  \leq 1,
\end{array}
\right.
\hspace{0.1in}
\phi(x,1,t)=
\left \{
\begin{array}{ll}
0.5-\frac{5}{8}x  & 0 \leq x  \leq 0.2,\\
                0 & 0.8 \leq x  \leq 1,\\
\end{array}
\right.
\end{equation*}
 \[
 \phi(0,y,t)=0.5,\quad \quad \phi(1,y,t)=0,
 \]
 and
 \begin{equation*}
\psi(x,0,t)=
\left \{
\begin{array}{ll}
0  &  0 \leq x  \leq 0.2,\\
\frac{-1}{8}+\frac{5}{8}x & 0.2 \leq x  \leq 1,
\end{array}
\right.
\hspace{0.1in}
\psi(x,1,t)=
\left \{
\begin{array}{ll}
0  & 0 \leq x  \leq 0.8,\\
-2+2.5 x & 0.8 \leq x  \leq 1,\\
\end{array}
\right.
\end{equation*}
 \[
 \psi(0,y,t)=0,\quad \quad \psi(1,y,t)=0.5,
  \]

 Figure \ref{boundaries} shows boundary values of $u$ and $v.$ In Figure \ref{contours} the contours plot of  solutions $u(x,y)$ and $v(x,y)$ are presented.

\begin{figure}
{\includegraphics[width=.7\textwidth]{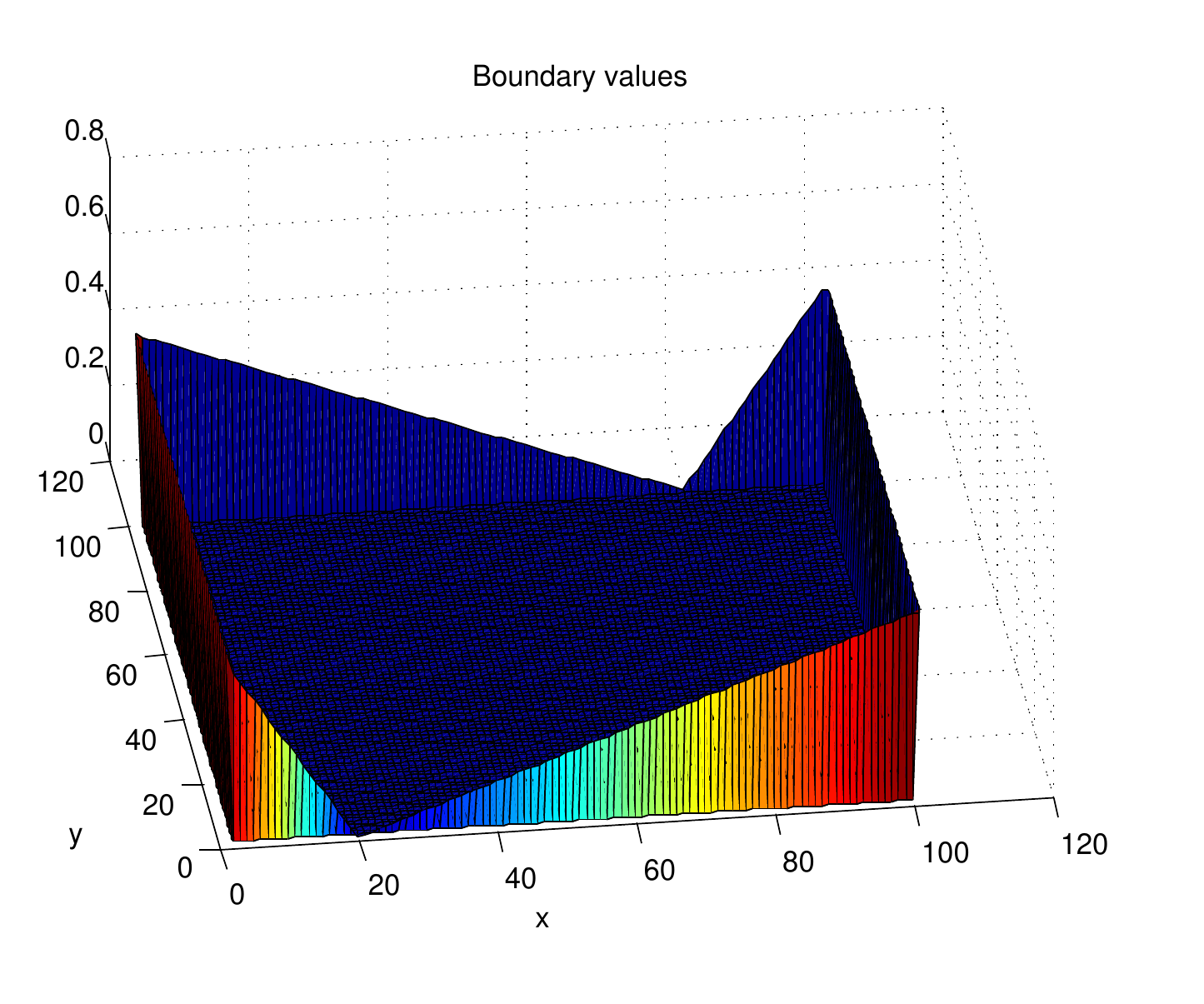}}
\centering
  \caption{Boundary values for $u(x,y);$ left $v(x,y)$; in the right.}\label{boundaries}
\end{figure}

 \begin{figure}
 \begin{center}
  \subfloat[contours of $u$]{\includegraphics[width=.5 \textwidth]{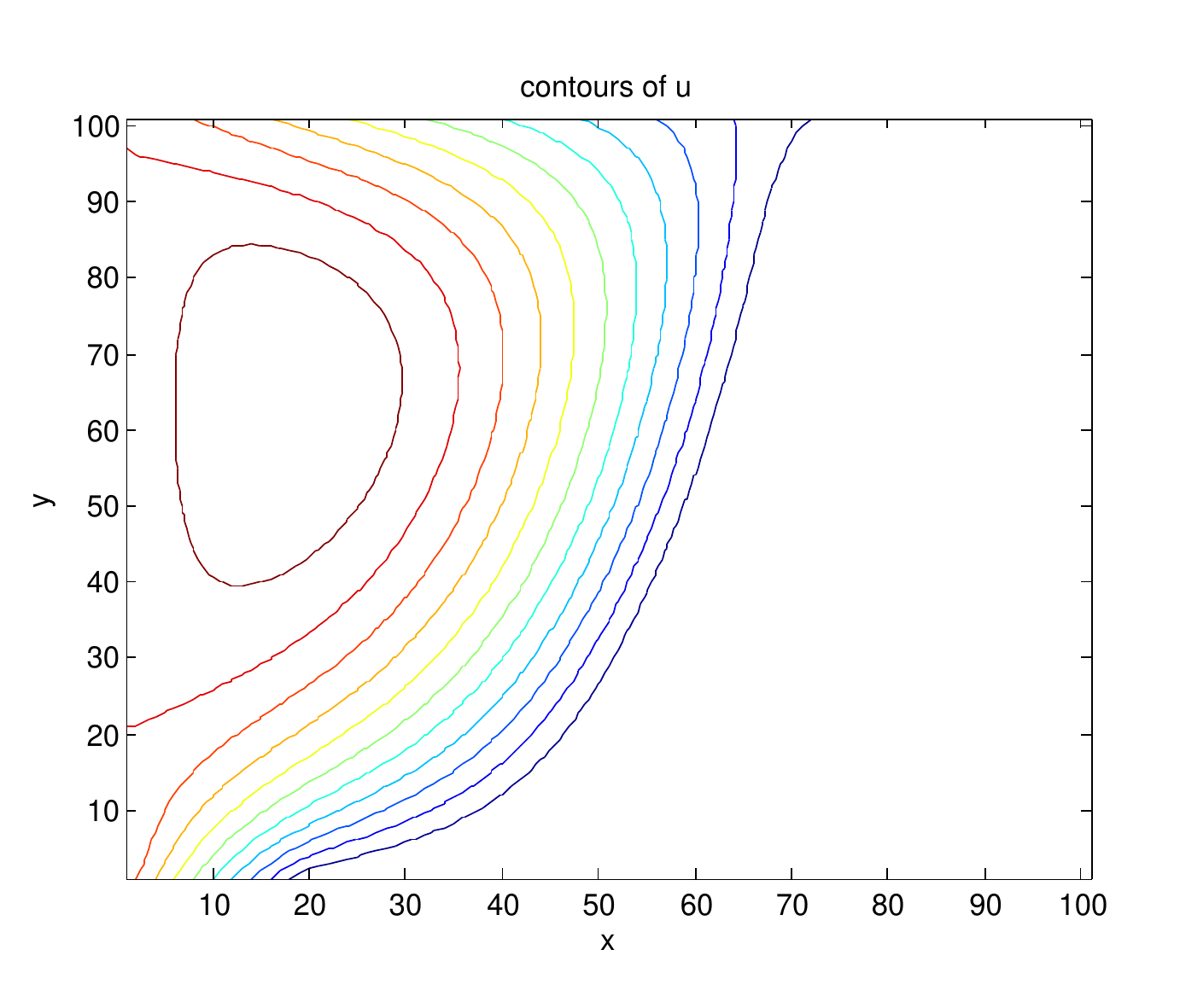}}
  \subfloat[contours of $v$]{\includegraphics[width=.487\textwidth]{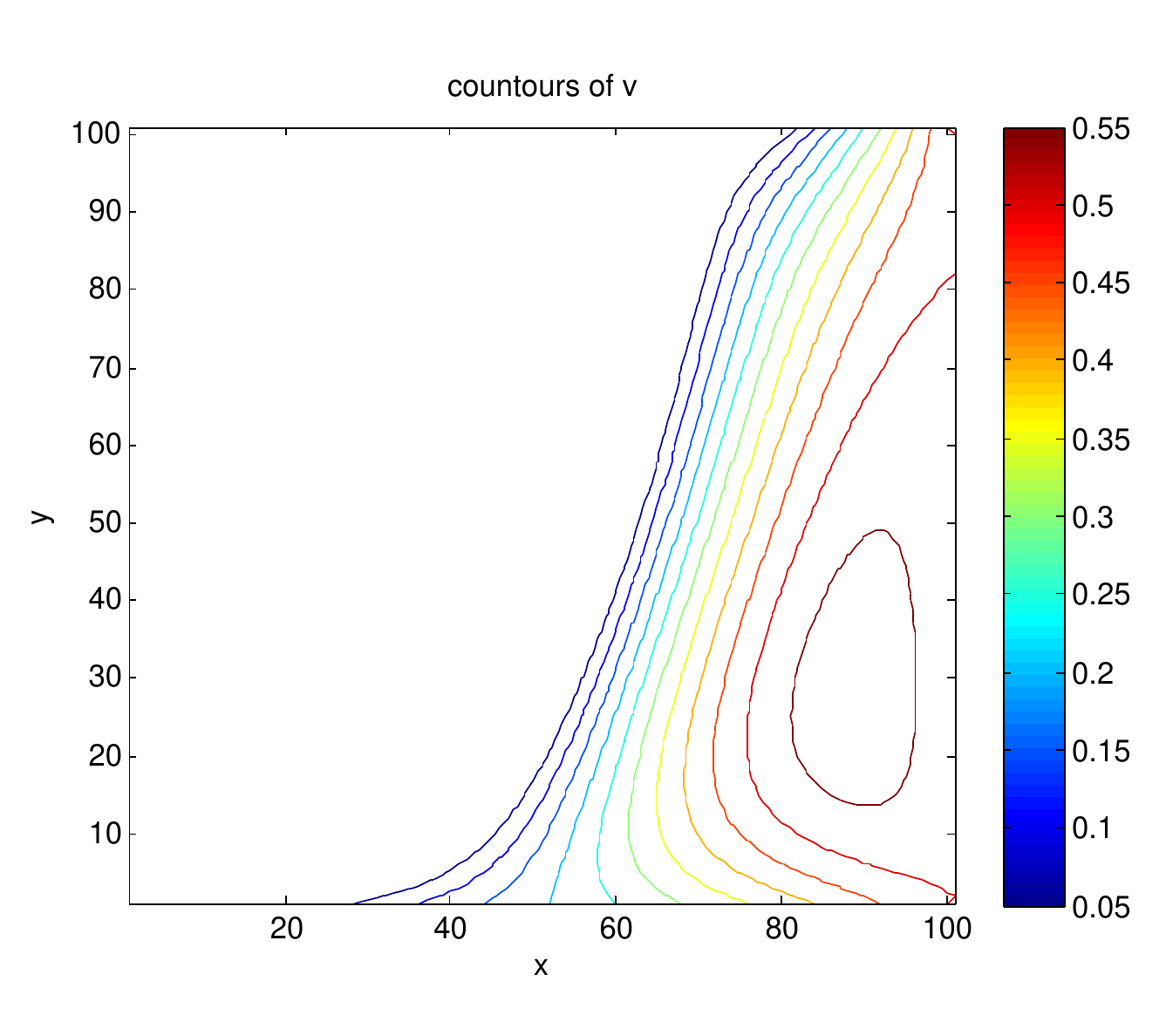}}
  \end{center}
  \caption{Contours of $u$ and $v.$}\label{contours}
 \end{figure}
\end{example}



\newpage

\vskip 50pt
\bibliographystyle{acm}%

\bibliography{farid}

\end{large}

\end{document}